\documentclass[oneside,english,12pt]{amsart}
\usepackage{}
\usepackage{txfonts}
\usepackage{amsfonts}
\usepackage{lmodern}

\usepackage[T1]{fontenc}
\usepackage[latin9]{inputenc}
\usepackage{geometry}
\usepackage{amsthm}
\usepackage{amssymb}
\usepackage{amssymb,amsmath,amsthm,tikz}
\usetikzlibrary{matrix,arrows,decorations.pathmorphing}
\usepackage{enumerate}
\usepackage{caption}

\usepackage[bookmarks=true, bookmarksopen=true,%
bookmarksdepth=3,bookmarksopenlevel=2,%
colorlinks=true,%
linkcolor=blue,%
citecolor=blue,%
filecolor=blue,%
menucolor=blue,%
urlcolor=blue]{hyperref}

\title{Radon transform for sheaves}
\author{Honghao Gao}
\address{Institut Fourier \\ Universit{\'e} Grenoble Alpes}
\email{honghao.gao@univ-grenoble-alpes.fr}

\makeatletter

\numberwithin{equation}{section}
\numberwithin{figure}{section}

\theoremstyle{plain}
\newtheorem{thm}{Theorem}[section]
\newtheorem{lem}[thm]{Lemma}

\newtheorem{prop}[thm]{Proposition}

\theoremstyle{definition}
\newtheorem{defn}[thm]{Definition}
\newtheorem{eg}[thm]{Example}

\theoremstyle{remark}
\newtheorem{rmk}[thm]{Remark}

\newcommand{\bbP}{{\mathbb{P}}}

\newcommand{\bbR}{{\mathbb{R}}}

\newcommand{\bbZ}{{\mathbb{Z}}}

\newcommand{\cA}{{\mathcal{A}}}

\newcommand{\cF}{{\mathcal{F}}}
\newcommand{\cG}{{\mathcal{G}}}
\newcommand{\cH}{{\mathcal{H}}}

\newcommand{\cP}{{\mathcal{P}}}

\newcommand{\del}{{\partial}}

\newcommand{\la}{{\langle}}
\newcommand{\ra}{{\rangle}}

\tikzset{node distance=1.5cm, auto}

\usepackage{babel}

\begin{document}
\begin{abstract}
We define the Radon transform functor for sheaves and prove that it is an equivalence after suitable microlocal localizations. As a result, the sheaf category associated to a Legendrian is invariant under the Radon transform. We also manage to place the Radon transform and other transforms in microlocal sheaf theory altogether in a diagram.
\end{abstract}
\maketitle

\section{Introduction}

The goal of the paper is to define the Radon transform for microlocal sheaf categories and study its properties.
The term ``microlocal'' refers to the consideration of the cotangent bundle when we study sheaves over a smooth manifold. This method was introduced in \cite{KS1} and has been systematically developed ever since \cite{KS}. The geometric nature of cotangent bundles makes microlocal sheaf theory a handy tool for problems in symplectic and contact geometry, such as those related to Lagrangian and Legendrian invariants \cite{GKS, Ta}, the Fukaya category \cite{Na, NZ}, Legendrian knots \cite{STZ, NRSSZ}, and more.

We define a sheaf version of the classical Radon transform (Definition \ref{RTfSC}). The classical Radon transform is an integral functional which takes a rapidly decreasing function on Euclidean space to a function on the space of hyperplanes in the Euclidean space. These analytic concepts have sheaf counterparts, where the sectional integration is captured by taking compactly supported sections. The Radon transform functor is expected to be an equivalence of categories, which corresponds to the reconstruction property of the classical Radon functional. In other words, the Radon transform should admit an inverse. In the first main theorem, we prove such an equivalence holds after some microlocal localizations.

\begin{thm}[Theorem \ref{RTn}]
The Radon transform induces an equivalence of categories:
$$\Phi: D^b(\bbR^n,\dot{T}^*\bbR^n) \rightarrow D^b(S^{n-1}\times \bbR, T^{*,+}(S^{n-1}\times \bbR)).$$
\end{thm}

We explore two directions based on this equivalence.

For the first direction, we study the Radon transform functor and the categorical invariants for conic Lagrangians or Legendrians. Given a Legendrian in the cosphere bundle, the subcategory of sheaves microsupported along the Lagrangian cone of the Legendrian is invariant under homogeneous Hamiltonian isotopies -- a result proven by Guillermou-Kashiwara-Schapira \cite{GKS}. Hence such categories define Legendrian isotopy invariants. We are particularly interested in knots (and links) in Euclidian space, and the associated Legendrian conormal bundles. The ambient contact space of a Legendrian conormal bundle admits two natural partial compactifications, each producing a categorical invariant. These categories are related by the Radon transform.

\begin{thm}[Proposition \ref{RTapplytoknotconormal} and Proposition \ref{simpleRTknotinv}]
Let $K\subset \bbR^3$ be a knot or a link. Let $\Lambda_K\subset T^\infty \bbR^3$ be the Legendrian knot conormal and $\Lambda_K'\subset T^{\infty}(S^2\times \bbR) \cong T^\infty \bbR^3$ be the counterpart. The Radon transform induces an equivalence of categories
$$D^b_{\Lambda_K}(\bbR^3)/ Loc(pt) \cong D^b_{\Lambda_K'}(S^2\times \bbR)/ Loc(S^2).$$
Moreover, the transform preserves the simpleness.
\end{thm}

This result holds in general for other compact Legendrians in the cosphere bundle of the Euclidean space (Theorem \ref{Equivuptoloc}).

As an application in the case of knot and link conormals, the equivalence builds towards the augmentation-sheaf correspondence. Augmentations are knot invariants defined over the Chekanov-Eliashberg differential graded algebra of the Legendrian knot conormal \cite{Ng, EENS1}. It is generally expected from the Nadler-Zaslow correspondence  \cite{Na, NZ} that augmentations correspond to simple sheaves. In particular, simple sheaves on $\bbR^3$  corresponds to augmentations \cite{Ga2}. It is also necessary to consider sheaves on $S^2\times \bbR$ within the framework of the Nadler-Zaslow correspondence. Our result here connects the two sheaf categories over different ambient spaces. More details on the story can be found in \cite{Ga}.

\smallskip

\medskip
For the second direction, we investigate the relations among some transforms in the literature of microlocal sheaf theory, together with the Radon transform in this paper. Here is a shortlist of the transforms involved:
\begin{itemize}
\item
Fourier-Sato transform. It is a fundamental transform studied by Kashiwara-Schapira \cite[Chapter 3.7]{KS}.
\item
Projective duality. It is a classical transform between the projective space and its dual. For real projective spaces, \cite{Sc, MT1, MT2, MT3} studied the constructible functions and the geometry of the subvarieties under the transform. \footnote{There is a distinction between our convention and that of \textit{loc.cit}: what we will call projective duality is called the Radon transform in \textit{loc.cit}, whereas the projective duality in \textit{loc.cit} refers to the underlining geometric transform.}
\item
Spherical duality. It can be regarded as a lift of the projective duality, up to the choice of a sign. The non-localized version was first studied in \cite{SKK}.
\item
Fourier-Tamarkin transform. It was first introduced by Tamarkin in \cite{Ta}, simply named ``Fourier transform''. This transform is applied in many other works, such as \cite{Ch, JT}.
\end{itemize}
The transforms above can be organized into a diagram where the connections among them are exhibited as restriction and polarization defined in Section \ref{Sec:Transforms}.
\begin{thm}\label{MainThm3}
We have the following diagram. The horizontal arrows are restrictions (Definition \ref{resdef}), and the vertical arrows is polarization (Definition \ref{polardef}).
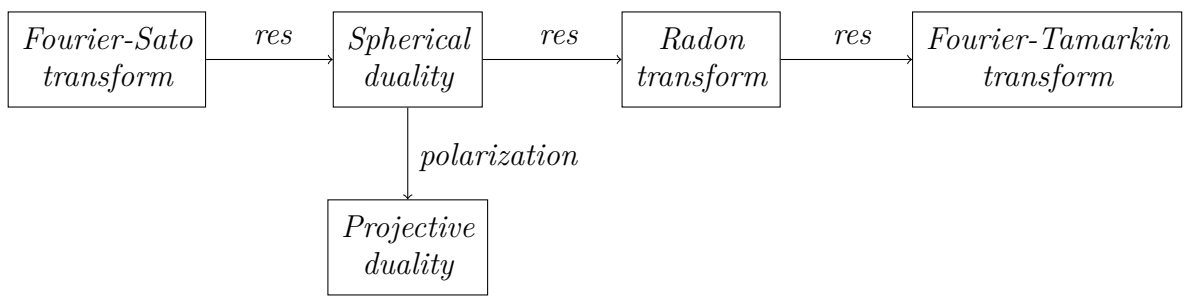
\begin{figure}[h!]
\begin{center}
\begin{tikzpicture}[
	block/.style = {draw, align=center,minimum height=3em, inner sep=5},
	arrow/.style = { semithick,}]
\node[block] (FS) at (-4, 0) {{Fourier-Sato}\\{transform}};
\node[block] (Spherical) at (0, 0) {{Spherical}\\{duality}};
\node[block] (Radon) at (3.9, 0) {{Radon}\\{transform}};
\node[block] (FT) at (8.5, 0) {{Fourier-Tamarkin}\\{transform}};
\node[block] (Projective) at (0, -2.5) {{Projective}\\{duality}};
\node[] (R1) at (-1.8, 0.3) {res};
\node[] (R2) at (2, 0.3) {res};
\node[] (R3) at (5.9, 0.3) {res};
\node[] (P1) at (1.2,-1.3) {polarization};
\draw[->] (FS) to (Spherical);
\draw[->] (Spherical) to (Projective);
\draw[->] (Spherical) to (Radon);
\draw[->] (Radon) to (FT);
\end{tikzpicture}
\end{center}
	\label{relations}
\end{figure}
\end{thm}

The rest of the paper is organized in a simple way. In Section \ref{Sec:Background} we review the microlocal sheaf theory. In each one of Sections \ref{Sec:Radon} -- \ref{Sec:Transforms}, we prove one the of main theorems of the paper.

\smallskip
\noindent\textbf{Acknowledgements.} We would like to thank Eric Zaslow for initiating the problem and advising. We thank Stéphane Guillermou and Pierre Schapira for valuable comments and suggestions for outreaching directions. We also thank Dmitry Tamarkin and Vivek Shende for helpful discussions. An earlier version of the work constitutes to the author's PhD dissertation at Northwestern University. This work is partially supported by the ANR projection ANR-15-CE40-0007 ``MICROLOCAL''.

\section{Background}\label{Sec:Background}
\subsection{Cotangent bundle}
Let $X$ be a real smooth manifold. Its cotangent bundle $T^*X$ is an exact symplectic manifold. Let $(x_1,\dotsb, x_n)$ be local coordinates on $X$ and $(\xi_1,\dotsb,\xi_n)$ be the induced coordinates along the fibre. Then the symplectic form is $\omega =\sum_{i=1}^n dx_i\wedge d\xi_i$, the canonical one-form is $\beta = \sum_{i=1}^n\xi_idx_i$, and $\omega = -d\beta$.

The \textit{antipodal map} on the cotangent bundle is $a: T^*X\rightarrow T^*X, a(x,\xi) = (x,-\xi).$ Let $\Omega\subset T^*X$ be an open subset. The antipodal set is $\Omega^a := \{(x,\xi)\in T^*X \,|\, (x,-\xi) \in \Omega\}.$ With induced symplectic forms from $T^*X$, $\Omega$ and $\Omega^a$ are anti-symplectic.

Suppose $M\subset X$ is a submanifold. The cotangent bundle is denoted by $T^*_MX$. Denote by $0_X$ the zero section of the cotangent bundle and denote by $\dot{T}^*X := T^*X\setminus 0_X $ the cotangent bundle removing zero section. For a subset $Z\subset T^*X$, denote $\dot{Z} := Z\cap \dot{T}^*X$.

Let $\bbR_+$ be the group of positive real numbers. This group has a canonical action on a vector bundle by dilation along the fiber direction. The action fixes each point on the zero section, and acts freely on the complement of the zero section.

\begin{defn}
Let $X,Y$ be two manifolds of the same dimension, and $\Omega_X,\Omega_Y$ be two open subsets in $T^*X$ and $T^*Y$. Let $p_1,p_2$ be projections from $X\times Y$ to the first and second components respectively. Given a smooth conic Lagrangian $L \subset \Omega_X^a\times \Omega_Y$.  If
$$p_1^a|_{L}: L \rightarrow \Omega_X,\quad\textrm{and}\quad p_2|_{L}: L \rightarrow \Omega_Y$$
are diffeomorphisms, then the composition 
$$\chi:=p_2|_{L}\circ (p_1^a|_{L})^{-1}: \Omega_X\rightarrow \Omega_Y$$
is a \textit{contact transform}.
\end{defn}

We remark that a contact transform is a symplectomorphism. The canonical choice of the symplectic form on $(T^*X)^a\times T^*Y$ is $(p_1^a)^{-1}\omega_{(T^*X)^a} + p_2^{-1}\omega_{T^*Y}$. This form is zero over the Lagrangian $L$. Therefore one has $\chi^*(\omega_{\Omega_Y}) = -\omega_{(\Omega_X)^a} = \omega_{\Omega_X}.$

\subsection{Sheaves} We review some concepts in \cite{KS}. Fix a ground field $k$. Suppose $X$ is a smooth manifold, let $D^b(X)$ be the bounded (dg) derived category of sheaf of $k$-modules. The classical derived category is the homological category of the dg derived category. All statements in this paper hold for either version of the derived categories.

\subsubsection{Singular support}
The key concept in the microlocal sheaf category is the \textit{singular support} (or \textit{micro-support}) $SS(\cF)\subset T^*X$ of a sheaf $\cF\in D^b(X)$. We recall the definition:
\begin{defn}[\cite{KS}, Definition 5.1.2]
Let $X$ be a manifold, $\cF\in D^b_{dg}(X)$ and $p = (x,\xi)\in T^*X$. The \textit{singular support} of $\cF$, denoted by $SS(\cF)$, is the subset of $T^*X$ satisfying: $p\notin SS(\cF)$ if and only if there exists an open neighborhood $U$ of $p$ such that for $x_0\in X$ and any $C^1$-real function $\phi$ in a neighborhood of $x_0$, with $\phi(x_0)=0, d\phi(x_0) \in U$, there is
$$(R\Gamma_{\{\phi(x)\geq 0\}}(\cF))_{x_0} = 0.$$
\end{defn}

The singular support $SS(\cF)$ is always a closed conic subset in $T^*X$, and more strongly, involutive \cite{KS}. It is hard to compute in general. Here are some examples.

\begin{eg}\label{SSexample}
Let $X$ be a smooth manifold of dimension $n$.
\begin{enumerate}
\item
If $M$ a smooth closed submanifold, then $SS(k_M) = T_M^*X$.
\item 
Let $\phi(x)$ be a smooth function on $X$ such that the level set $\{\phi(x) = 0\}$ is a closed manifold of dimension $n-1$, and that $d\phi(x)\neq 0$ on $\{\phi(x) =0\}$. Then $S = \{\phi \geq 0\}$ is a closed submanifold of dimension $n$ with a smooth boundary. There is
\begin{equation}\label{closedSS}
{SS}(k_S)= \{(x,\lambda d\phi(x))\,|\, \lambda\phi(x) =0, \lambda\geq 0, \phi(x) \geq 0\}.
\end{equation}
Let $U = \{\phi > 0\}$, then
$${SS}(k_U)= \{(x,\lambda d\phi(x))\,|\, \lambda\phi(x) =0, \lambda\leq 0, \phi(x) \geq 0\}.$$
\end{enumerate}
\end{eg}

The singular support satisfies a triangle inequality. Let $\cF_1\rightarrow \cF_2\rightarrow \cF_3 \xrightarrow{+1}$ be a distinguished triangle in $D^b(X)$, there is
\begin{equation}\label{triangleinequality}
SS(\cF_i) \subset SS(\cF_j)\cup SS(\cF_k),
\end{equation}
for any permutation $\{i,j,k\} = \{1,2,3\}$.

The singular support also has the following functorial properties.

\begin{prop} [\cite{KS}, Proposition 5.4.4 and Proposition 5.4.5]\label{SSfunctorial}
Let $f: Y\rightarrow X$ be a smooth map between manifold. It induces the bundle morphisms:
\begin{equation}\label{SStransform}
T^*Y \xleftarrow{{f_d}} Y\times_XT^*X \xrightarrow{f_\pi} T^*X.
\end{equation}

\begin{enumerate}
\item
Let $\cG\in D^b_{}(Y)$. Assume that $f$ is proper on $supp(\cG)$, then
\begin{equation}\label{SSfunctorial1}
SS(Rf_*\cG) \subset f_\pi({f_d^{-1}} SS(\cG))).
\end{equation}
\item[(1')]
In addition, if $f$ is a closed embedding, the inclusion is an equality.

\item
For any $\cF\in D^b_{}(X)$,
\begin{equation}\label{SSfunctorial2}
SS(f^{-1}\cF) =f_d(f_\pi^{-1} SS(\cF))).
\end{equation}
\end{enumerate}

\end{prop}

Using singular support, one can associate a category of sheaves to a subset in the cotangent bundle. Let $V\subset T^*X$ be a subset, we define the full subcategory $D^b_{V}(X)\subset D^b(X)$ to be
$$D^b_{V}(X) := \{\cF\in D^b_{}(X) \,|\, SS(\cF) \subset V\}.$$

\subsubsection{Microlocaliztion}

A manifold $X$ with an $\bbR_+$ action is called a \textit{conic space}. A sheaf $\cF\in D^b(X)$ on a conic space is called a \textit{conic sheaf} if the restriction to each orbit is a constant sheaf. Let $D^b_{conic}(X)\subset D^b(X)$ be the full subcategory of conic sheaves. Real vector bundles are examples of conic spaces where $\bbR_+$ acts as dilation along the fibers.

The microlocalization functor was first introduced by Sato \cite{Sa}. Let $M$ be a closed submanifold of $X$, the microlocalization functor along $M$
$$\mu_M: D^b(X)\rightarrow D^b_{conic}(T_M^*X),$$
captures the singular support along $T_M^*X$. It was later generalized to the $\mu hom$ bifunctor by Kashiwara-Schapira \cite{KS1}:
$$\mu hom:D^b_{}(X)^{op}\otimes D^b_{}(X)\rightarrow D^b_{conic}(T^*X).$$
The $\mu hom$ functor detects singular support, namely $supp(\mu hom(\cF,\cF)) = SS(\cF)$.

\subsubsection{Convolution}

Let $X,Y$ be two smooth manifolds, and let $\pi_1,\pi_2$ be the projections from $X\times Y$ to the first and second components. One can associate to a sheaf $K\in D^b(X\times Y)$ a \textit{convolution} functor $\Phi_K: D^b(X) \rightarrow D^b(Y)$,
\begin{align}
\begin{split}
\Phi_K(\cF) = R\pi_{2!}\left(K\otimes \pi_1^{-1}\cF\right).
\end{split}
\end{align}
The convolution has a right adjoint $\Psi_K(\cG) = R\pi_{1*} R\cH om(K, \pi_2^!\cG)$. The sheaf $K$ is called the \textit{kernel} of the convolution.

\subsubsection{Quantized contact transform}

Suppose $\Omega\subset T^*X$ is a subset. Define a localized category
$D^b_{}(X,\Omega)$ to be the quotient (in the dg setting, the quotient is in the sense of \cite{Dr}),
$$D^b_{}(X,\Omega) := D^b_{}(X)/ D^b_{T^*X\setminus \Omega} (X).$$
Objects in $D^b_{}(X,\Omega)$ are the same as that in $D^b_{}(X)$. If a morphism $f:\cF\rightarrow \cG$ in $D^b(X)$ satisfies 
$$SS(Cone(f))\cap \Omega = \emptyset,$$
it is considered as an isomorphism in $D^b_{dg}(X,\Omega)$. Morphisms in $hom_{D^b_{}(X,\Omega)}(\cF, \cG)$ are roofs of the form $\{\cF\xleftarrow{\sim} \cH\rightarrow \cG\}$, and compositions are given by the pull back diagrams.

\begin{defn}\label{QCT}
Let $X$ and $Y$ be two manifolds of the same dimension. Suppose $\Omega_X\subset \dot{T}^*X$ and $\Omega_Y\subset \dot{T}^*Y$ are two open subsets, and $\chi: \Omega_X \rightarrow \Omega_Y$ is a contact transform.

A sheaf $K\in D^b(X\times Y)$ is a \textit{quantized contact transform} from $(X,\Omega_X)$ to $(Y,\Omega_Y)$ if:
\begin{enumerate}
\item $K$ is cohomologically constructible, \footnote{A technical definition, see \cite[Definition 3.4.1]{KS}.} 
\item $\big((p_1^a)^{-1}(\Omega_X)\cup p_2^{-1}(\Omega_Y)\big)\cap SS(K) \subset L$,
\item the natural morphism $k_{L}\rightarrow \mu hom(K,K)|_L$ is an isomorphism in $D^b(L)$.
\end{enumerate}
\end{defn}

\begin{thm}[\cite{KS}, Theorem 7.2.1]\label{ContacttransformKS}
Suppose $K\in D^b(X\times Y)$ is a quantized contact transform from $(X,\Omega_X)$ to $(Y,\Omega_Y)$. Then
\begin{enumerate}
\item
The convolution
$$\Phi_K: D^b(X,\Omega_X)\rightarrow D^b(Y,\Omega_Y)$$
is an equivalence of categories. 
\item
If $\cF_1,\cF_2$ are two objects in $D^b(X,\Omega_X)$, there is a natural isomorphism in $D^b(Y,\Omega_Y)$:
$$\chi_* \mu hom(\cF_1,\cF_2) \cong \mu hom(\Phi_K(\cF_1), \Phi_K(\cF_2)).$$
\end{enumerate}

\end{thm}

\subsubsection{Simple sheaves}
Let $p = (x,\xi) \in \dot{T}^*X$. A \textit{test function} of $p$ is a locally defined function $\phi$ is a neighborhood of $x$ such that $\phi(x)=0$ and $d\phi(x) = \xi$. Further let $\Lambda\subset \dot{T}^*X$ be a closed conic Lagrangian and assume $p\in \Lambda$. A test function $\phi$ of $p$ is \textit{transverse} if $\Gamma_{df}$ intersects $\Lambda$ transversely at $p$. A sheaf $\cF$ is simple at $p$ if it satisfies one of the following equivalent conditions:
\begin{enumerate}
\item
There is a transverse test function $\phi$ of $p$ such that
$R\Gamma_{\{\phi\geq 0\}}(\cF)_x \cong k[d]$
for some integer $d$.
\item 
$\mu hom(\cF,\cF)_p \cong k.$
\end{enumerate}
The sheaf $\cF$ is \textit{simple} along $\Lambda$ if it is simple at every point on $\Lambda$.

\begin{lem}\label{simplelemma}
Let $X$ be a smooth manifold and $S$ a smooth submanifold with a smooth boundary of codimension $1$. The constant sheaf $k_S$ supported on $S$ is simple along its singular support. 
\end{lem}

The lemma follows a short argument using microlocal techniques: $k_S$ is microlocally isomorphic to the constant sheaf on the boundary, which is simple. We provide a pedestrian proof using the first definition above.
\begin{proof}
Because the singular support is locally defined and the simpleness is a local property, one can assume $X =\bbR^n$ with coordinates $(x_1,\dotsb, x_n)$, and $S$ is the upper half space $\{x_n\geq 0\}$. Suppose $(x,\xi)$ are the canonical coordinates on $T^*X$, by Example \ref{SSexample}, 
$$\Lambda := SS(k_S)=\{(x_i,\xi_i)\in T^*\bbR^n \,|\,  x_n=0, \xi_n\geq 0,\textrm{ and }\xi_i =0\textrm{ for }i\neq n\}.$$

For $p=(0, \lambda dx_n)$ and $R>0$, we consider a test function $\phi(x_i)= \lambda x_n + \sum_{i=1}^{n-1}Rx_i^2$. The graph of its differential is $\Gamma_{d\phi} = \{\xi_n=\lambda, \xi_i= 2x_i \textrm{ for } i\neq n\}$. The tangent spaces of $\Lambda$ and $\Gamma_{d\phi}$ at $p$ are
\begin{align*}
T_p\Lambda &= \textrm{Span}\{\del_{x_i} \text{ for } i\neq n, \del_{\xi_n}\}; \\
T_p\Gamma_{d\phi} &= \textrm{Span}\{\del_{x_i} + 2R\del_{\xi_i} \text{ for } i\neq n, \del_{x_n}\}.
\end{align*}
Hence $\Lambda\cap \Gamma_{d\phi}$ transversely at $p$.

Finally we compute $R\Gamma_{\{\phi\geq 0\}}(k_S)_0$. The stalk is a directed limit of $R\Gamma(U,R\Gamma_{\{\phi\geq 0\}}(k_S))$. Choose the opens to be a nested family $\{U_n\}$ of cubes. Each of $(U_n, S\cap U_n)$ equals to $(X,S)$ by a dilation in coordinates. Apply $RHom(-,k_S)$ to the canonical short exact sequence
$$0 \rightarrow k_{\{\phi<0\}}\rightarrow k_{X} \rightarrow k_{\{\phi\geq 0\}}\rightarrow 0,$$
we get
\begin{align*}
R\Gamma(X,R\Gamma_{\{\phi\geq 0\}}(k_S)) &= RHom ( k_{\{\phi\geq 0\}}, k_S)  \\
			&= Cone(RHom (k_X, k_S)\rightarrow RHom( k_{\{\phi<0\}},k_S)) \\
			&= Cone(R\Gamma(X,k_S)\rightarrow R\Gamma(\{\phi<0\},k_S)) \\
			&= Cone(k\rightarrow 0) = k[1].
\end{align*}
It is the same result for any neighborhood of $0$ in the nested family. By passing to the direct limit, $R\Gamma_{\{\phi\geq 0\}}(k_S)_0 \cong k[1]$. It is a similar calculation for other points on $\Lambda$. We conclude that $k_S$ is simple.
\end{proof}

\section{Radon transform}\label{Sec:Radon}

\subsection{Radon transform} A hyperplane in $\bbR^n$ is characterized as
$$P_{(\hat{n}, r)}: = \{\vec{x}\in \bbR^n \,|\, \vec{x}\cdot \hat{n} = r\},$$
where $\hat{n}$ is a unit length vector representing a normal direction of the plane, and $r$ is the signed distance from the plane to the origin. There is a symmetry between the planes $P_{(\hat{n}, r)} = P_{(-\hat{n}, -r)}$, but we do not intend to identify them. The space of hyperplanes is diffeomorphic to $S^{n-1}\times \bbR$, with coordinates $(\hat{n},r)$.

Let $f$ be a fast decreasing function or a compactly supported smooth function on $\bbR^n$, its Radon transform $R(f)$ is a function on $S^{n-1}\times \bbR$:
$$R(f)(\hat{n},r) = \int_{P_{(\hat{n}, r)}} f.$$
The integration of $R(f)$ along the seconding component is
\begin{equation}\label{classicalintRadon}
\int_{-\infty}^r  R(f)(\hat{n},s)ds = \int_{\{\vec{x}\, \cdot\, \hat{n} \leq r\}} f(\vec{x}) d\vec{x}.
\end{equation}

We define the Radon transform for sheaves. 

\begin{defn} \label{RTfSC}
Let $\pi_1, \pi_2$ be the projections from $\bbR^n\times S^{n-1}\times\bbR$ as follows:
\begin{center}
\begin{tikzpicture}
  \node (A){$\bbR^n\times (S^{n-1}\times  \bbR)$};
  \node (C)[below of=A,node distance=1.5cm] {};
  \node (D)[left of=C,node distance=2.5cm] {$\bbR^n$};
  \node (F)[right of=C,node distance=2.5cm]{$S^{n-1}\times \bbR$};
  \draw[->] (A) to node [swap] {$\pi_{1}$} (D);
  \draw[->] (A) to node []{$\pi_2$}(F);
\end{tikzpicture}
\end{center}
Let $k_{\{\vec x \cdot \hat n\leq r\}}$ be the constant sheaf supported on $\{\vec x\cdot \hat n\leq r\}\subset \bbR^n\times S^{n-1}\times \bbR$. The Radon transform for sheaves is defined to be 
\begin{align}
\begin{split}
\Phi: D^b(\bbR^n) &\rightarrow D^b(S^{n-1}\times \bbR), \\
		\cF &\mapsto R\pi_{2!}(k_{\{ \vec{x}\, \cdot\, \hat{n}\leq r\}}\otimes \pi_1^{-1}\cF).
\end{split}
\end{align}

It is left adjoint to $\Psi: D^b(S^{n-1}\times \bbR) \rightarrow D^b(\bbR^n)$, $\Psi(\cG) = R\pi_{1*}R\cH om(k_{\{ \vec{x}\, \cdot\, \hat{n}\leq r\}}, \pi_2^{!}\cG).$
\end{defn}
\begin{rmk}
At a fixed point $(\hat{n},r)\in S^{n-1}\times \bbR$, the stalk of the transformed sheaf $\Phi(\cF)$ is the compactly supported sheaf cohomology of $\cF$ over a sublevel set:
$$\Phi(\cF)_{(\hat{n},r)} = R\Gamma_c(\{\vec{x} \cdot \hat{n}\leq r\},\cF).$$
The analogy between this expression and (\ref{classicalintRadon}) gives the name of the functor.
\end{rmk}

\subsection{Equivalence}

In this section, we prove:
\begin{thm}\label{RTn}
The Radon transform for sheaves induces an equivalence between the localized sheaf categories
\begin{align*}
\Phi: D^b(\bbR^n,\dot{T}^*\bbR^n) &\rightarrow D^b(S^{n-1}\times \bbR, T^{*,+}(S^{n-1}\times \bbR)),\\
		\cF &\mapsto R\pi_{2!}(k_{\{ \vec{x}\, \cdot\, \hat{n}\leq r\}}\otimes \pi_1^{-1}\cF).
\end{align*}
Here $T^{*,+}(S^{n-1}\times \bbR) = T^*S^{n-1}\times T^{*,+}\bbR$ whereas $T^{*,+}\bbR = \{(r,\eta_r)\in T^*\bbR\,|\, \eta_r >0\}$. 
\end{thm}

\begin{proof}
Let $K$ be the convolution kernel of the Radon transform. By bullet points, we organize the proof into parts of: (1) fixing conventions, (2) computing the contact transform determined by $K$, and (3) showing that $K$ satisfying the three properties being a quantized contact transform. By Theorem \ref{ContacttransformKS} we get the desired equivalence of categories.

\medskip
$\bullet$ Conventions.  
Let $\vec{x}\cdot \vec{y}$ be the standard flat metric on Euclidean space and let $\|\vec{x}\|$ be the standard norm. 

We use vectors as global coordinates.
A point on $T^*\bbR^n \cong \bbR^n\times (\bbR^n)^*$ is denoted by $p = (\vec{x},\vec{\xi})$, where $\vec{x}$ and $\vec{\xi}$ are $n$-tuple vectors. As to $T^*(S^{n-1}\times \bbR) = T^*S^{n-1}\times T^*\bbR$, we identify $S^{n-1}$ with the unit sphere in $\bbR^n$, and its cotangent bundle $T^*S^{n-1}$  with a sub-bundle of $T^*\bbR^n|_{S^{n-1}}$:
$$T^*S^{n-1}= \{(\hat{n},\vec{\eta})\in T^*\bbR^n \,|\, \|\hat{n}\|=1, \hat{n}\cdot \vec\eta =0\}.$$
Let $(r,\eta_r)$ be the coordinates on $T^*\bbR$.

Observe that the Radon transform is the convolution by $k_S$, where $S = \{\vec{x}\cdot  \hat{n} \leq r\}\subset \bbR\times (S^{n-1}\times \bbR)$. Let $\del S$ be the boundary of $S$. It is a smooth submanifold of codimension one.

\medskip
$\bullet$ Contact transform.
We first verify that the singular support of $K = k_{\{\vec{x}\cdot \hat{n}\leq r\}}$ induces a symplectomorphism
\begin{align}\label{Chiformula}
	\begin{split}
	\chi: \dot{T}^*\bbR^n &\rightarrow T^{*,+}(S^{n-1}\times \bbR), \\
		 (\vec{x},\vec\xi) &\mapsto \big(\,\frac{\vec\xi}{\|\vec\xi\|}\, ,\, \frac{ \vec{x}\cdot \vec\xi}{\|\vec\xi\|}\, ,\, -\|\vec\xi\|\vec{x} + \frac{\vec\xi (\vec{x}\cdot \vec\xi)}{\|\vec\xi\|} \, ,\, \|\vec\xi\|\,\big).
	\end{split}
\end{align}
We remark that the map is invertible,
$$\chi^{-1}(\hat{n}, r, \vec{\eta},\eta_r) = \big(-\frac{1}{\eta_r}\vec\eta + r\hat{n}\, ,\, \eta_r \hat{n}\big),$$
and preserves the dilation action -- for any positive real number $\lambda$,
$$
\chi^{-1} (\hat{n}, r, \lambda\vec{\eta},\lambda\eta_r)= \big(-\frac{1}{\lambda\eta_r}\lambda\vec\eta + r\hat{n}, \lambda\eta_r \hat{n}\big) = \lambda\cdot  \big(-\frac{1}{\eta_r}\vec\eta + r\hat{n}, \eta_r \hat{n}\big) = \lambda \cdot \chi^{-1} (\hat{n}, r, \vec{\eta},\eta_r).
$$

Define $L = \dot {SS}(k_S)\subset \dot{T}^*(\bbR^n\times S^{n-1}\times \bbR)$. Set $F(\vec{x},\hat{n},r) :=\vec{x}\cdot \hat n -r$ and $G(\vec{x},\hat{n},r):= \hat{n}\cdot \hat{n}-1$. By Example \ref{SSexample}, $L$ is contained in the conormal bundle of $\del S$ in $\bbR^n\times(S^{n-1}\times \bbR)$, its global coordinates are characterized by the following two constraints.
\begin{enumerate}
\item
The hypersurface $\del S\subset \bbR^n\times S^{n-1}$ is the level set $\{F=0 \}$ where $F(\vec{x},\hat{n},r) :=\vec{x}\cdot \hat n -r$. Hence the cotangent bundle of $\del S$ consists of covectors in $T^*(\bbR^n\times S^{n-1}\times \bbR)$ that are orthogonal to $dF$.
\item
We think of $\bbR^n \times S^{n-1}\times \bbR$ as a subspace of $\bbR^{2n+1}$. A covector being in $T^*(\bbR^n\times S^{n-1}\times \bbR)$ rather than just $T^*\bbR^{2n+1}$ requires this covector is orthogonal to $dG$, where $G(\vec{x},\hat{n},r):= \hat{n}\cdot \hat{n}-1$.
\end{enumerate}
A conormal vector is perpendicular to $dG$ and the orthogonal complement of $dF$. We can take the projection of $dF$ to the orthogonal complement of $dG$, 
$$dF - \frac{dG \la dF,dG\ra}{\|dG\|^2}.$$
Since $dF = \hat n\, d\vec x + \vec x \,d\hat n - dr$, and $dG = 2\hat n\, d\hat n$, it equals to
$$(\hat{n},\vec{x}, -1) - (0,2\hat {n}, 0 )\frac{2\hat n \cdot \vec x}{4} = (\hat{n}, \vec x-r\hat n,-1).$$

By (\ref{closedSS}), non-zero covectors in $SS(k_S)$ point to the interior. Therefore
$$L = \left\{\big((\vec{x},\hat{n},r),\ell(-\hat n, -\vec x + r\hat n, 1)\big)\in T^*(\bbR^n\times S^{n-1}\times \bbR) \,|\, \hat n\cdot \hat n =1,\vec x\cdot\hat n = r,\ell>0 \right\}.$$
Here the coordinates on $T^*(\bbR^n\times S^{n-1}\times \bbR)$ are in forms of $\big((\vec{x},\hat{n},r),(\vec{\xi}, \vec{\eta},\eta_r)\big)$.

Let $a_{T^*X}:(T^*X)^a \xrightarrow{\sim} T^*X$ be the antipodal map. For $X=\bbR^3$ and $Y=S^{n-1}\times \bbR$, we consider the identification
$$a_{T^*X}\times id_{T^*Y}: (T^*X)^a \times T^*Y \rightarrow T^*(X\times Y).$$
The Lagrangian $L$ sits inside $T^*(\bbR^n\times S^{n-1}\times \bbR)$. Project $L$ to each of $T^*\bbR^n$ and $T^*(S^{n-1}\times \bbR)$
\begin{align*}
p_1^a|_L: \big((\vec{x},\hat{n},r),\ell(-\hat n, -\vec x + r\hat n, 1)\big) &\mapsto (\vec{x}, \ell\hat{n}),\\
p_2|_L: \big((\vec{x},\hat{n},r),\ell(-\hat n, -\vec x + r\hat n, 1)\big) &\mapsto (\hat{n},r, -\ell\vec{x} + \ell r \hat n, \ell).
\end{align*}
By setting
$$(\vec{x},\vec\xi) = (\vec{x}, \ell\hat{n}), \qquad (\hat{n}, r, \vec{\eta},\eta_r) = (\hat{n},r, -\ell\vec{x} + \ell r \hat n, \ell),$$
we get equations
$$\vec\xi = \ell \hat{n}, \quad \vec\eta = -\ell \vec{x}+\ell r\hat{n},\quad \eta_r = \ell.$$
Combining the facts that $\hat{n}\cdot \hat{n} = 1$ and $\vec{x}\cdot \hat n = r$, we can solve everything in terms of $\vec{x}$ and $\vec\xi$:
$$\ell = \| \vec\xi \|, \; \hat n = \frac{\vec\xi}{\|\vec\xi\|},\;  r = \frac{\vec{x}\cdot \vec\xi}{\|\xi\|},\; \vec{\eta} = -\|\vec\xi\|\vec{x} + \frac{\vec\xi (\vec{x}\cdot \vec\xi)}{\|\vec\xi\|},\; \eta_r=\| \vec\xi \|.$$
Then the contact transform defined by $\Lambda$ is
$$p_2|_L\circ (p_1^a|_L)^{-1}: (\vec{x},\vec\xi) \mapsto \big(\,\frac{\vec\xi}{\|\vec\xi\|}\, ,\, \frac{ \vec{x}\cdot \vec\xi}{\|\vec\xi\|}\, ,\, -\|\vec\xi\|\vec{x} + \frac{\vec\xi (\vec{x}\cdot \vec\xi)}{\|\vec\xi\|} \, ,\, \|\vec\xi\|\,\big).$$
This is precisely $\chi$.

\medskip
$\bullet$ Quantization.
Next we verify that $k_S$ satisfies the three conditions in Definition \ref{QCT}.

\begin{enumerate}
\item
Since $S$ is closed, $k_S$ is constructible, and hence cohomologically constructible. (\cite{KS} Proposition 8.4.9 and Exercise 8.4.4).
\item
Let $\Omega_X = \dot{T}^*\bbR^n$ and $\Omega_Y =  T^{*,+}(S^{n-1}\times \bbR)$. Because $\Omega_X$ or $\Omega_Y$ do not contain points on the zero sections, neither do $\big( (p_1^a)^{-1}(\Omega_X)\cup p_2^{-1}(\Omega_Y)\big)$. By construction we have $L = \dot {SS}(k_S)$, and hence 
$$\big( (p_1^a)^{-1}(\Omega_X)\cup p_2^{-1}(\Omega_Y)\big)\cap SS(k_S) \subset L.$$
\item
$S$ is smooth manifold with a codimension $1$ boundary. By Lemma \ref{simplelemma}, $k_S$ is simple along $\dot{SS}(k_S)$. By the (second equivalent) definition of the simple sheave, it implies the canonical map
$$k_\Lambda\rightarrow \mu hom (k_S,k_S)|_\Lambda$$
is an isomorphism. 
\end{enumerate}
At this point, we have verified that the kernel of the Radon transform is a quantized contact transform from $(\bbR^n,\dot{T}^*\bbR^n)$ to $(S^{n-1}\times\bbR, T^{*,+}(S^{n-1}\times\bbR))$. Applying Theorem \ref{ContacttransformKS} we complete the proof.
\end{proof}

\section{Invariants}

By the work of Guillermou-Kashiwara-Schapira \cite{GKS}, a homogeneous Hamiltonian isotopy of the cotangent bundle with the zero section removed has a unique sheaf quantization. Therefore the subcategory of sheaves micro-supported along a conic Lagrangian is an invariant under a homogenous Hamiltonian isotopy, and the equivalence is given by the convolution of the quantized sheaf kernel.

A closed conic Lagrangian $\Lambda \in \dot{T}^*\bbR^n \cong T^{*,+}(S^{n-1}\times \bbR)$ admits two categorical invariants. We prove they are equivalent by the Radon transform, and apply the result to knot theory.

\subsection{GKS theorem}
Let $X$ be a smooth manifold and $I\subset \bbR$ an open interval containing the origin. A homogeneous Hamiltonian isotopy of $\dot{T}^*X$ is a smooth map $H: \dot{T}^*X\times I\rightarrow \dot{T}^*X$ such that $H_t:= H(-,t)$ satisties: (1) $H_0$ is the identity map, (2) $H_t$ is a symplectomorphism form each $t\in I$, (3) $H_t$ respects the dilation along the fiber. 

The GKS theorem \cite[Theorem 3.7]{GKS} states that each homogeneous Hamiltonian isotopy admits a unique sheaf quantization (up to a unique isomorphism), which is a locally bounded sheaf $K \in D(X\times X\times I)$ with: (1) the restriction to $X\times X\times \{0\}$ is isomorphic to the constant sheaf $k_\Delta$ on the diagonal $\Delta\subset X\times X$, (2) $\dot{SS}(K)$ is contained in a unique conic Lagrangian $L\in \dot{T}^*(X\times X\times I)$ \cite[Lamma, A.2]{GKS}, which has the property that  
$$L_t = \{(x,\xi), (x, -H(x, \xi, t))\,|\, (x,\xi)\in \dot{T}^*X, t\in I\}.$$

Convolution by the kernel $K$ restricted to any $t\in I$ is an autoequivalence of the category of locally bounded sheaves on $X$ \cite[Proposition 3.2]{GKS}. If in addition the Hamiltonian is the identity outside a ``horizontally compact region'' (compact after taking the quotient of the dilation) and the interval $I$ is relatively compact, then the category can be taken to be the bounded derived category (other than locally bounded) \cite[Remark 3.8]{GKS}.

The second property of the kernel indicates that the convolution deforms the singular support of sheaves according to the homogeneous Hamiltonian isotopy. Because a horizontally compact deformation of a closed conic Lagrangian extends to a homogeneous Hamiltonian isotopy in the ambient cotangent \cite[Proposition A.5]{GKS}, the convolution gives rise to an equivalence between the categories of sheaves microsupported along the conic Lagrangian before and after the deformation, stated as follows:

\begin{prop}\cite[Corollary 3.13]{GKS}
Suppose $\Lambda_0\subset \dot{T}^*X$ is a closed conic Lagrangian, $I$ is a relatively compact open interval containing the origin. Suppose $H: \Lambda_0\times I\rightarrow \dot{T}^*X$ is a deformation which is the identity outside a horizontally compact region. Let $H_t = H(-,t)$ and $\Lambda_t = H_t(\Lambda_0)$. There is an equivalence of categories
$$D^b_{\Lambda_0\cup 0_X}(X)\xrightarrow{\sim} D^b_{\Lambda_t\cup 0_X}(X).$$
\end{prop}

In other words, $D^b_{\Lambda_0\cup 0_X}(X)$ is an invariant under homogeneous Hamiltonian isotopies. When the context is clear, we simply write $D^b_{\Lambda_0}(X)$. Next we study how this category behaves under the Radon transform.

\subsection{Comparing invariants}\label{I2}
Let $X$ be a smooth manifold, $\Omega \subset T^*X$ a conic open subset and $\Lambda \subset \Omega$ a conic closed Lagrangian submanifold. Let $D^b_\Lambda(X,\Omega)$ be the full subcategory of $D^b(X,\Omega)$ consisting of sheaves $\cF$ with $\dot{SS}(\cF) \subset \Lambda$. Also recall $\chi: \dot{T}^*\bbR^n\rightarrow T^{*,+}(S^{n-1}\times \bbR)$ from (\ref{Chiformula}), the contact transform induced from the Radon transform $\Phi$ .

\begin{prop}\label{RadonSh}
Suppose $\Lambda \subset \dot{T}^*\bbR^n$ is a closed conic Lagrangian submanifold. Let $\Lambda' = \chi(\Lambda)$. The Radon transform $\Phi$ induces an equivalence of categories:
$$D^b_{\Lambda}(\bbR^n,\dot{T}^*\bbR^n)\xrightarrow{\sim} D^b_{{\Lambda'}}(S^{n-1}\times \bbR, T^{*,+}(S^{n-1}\times \bbR)).$$
\end{prop}

\begin{proof}
We show that any sheaf with $\dot{SS}(\cF)\subset \Lambda$ satisfies $\dot{SS}(\Phi(\cF))\subset \Lambda'$. Let $p_1$ (resp. $p_2$) be the projection from $\dot{T}^*\bbR^n\times \dot{T}^*(S^{n-1}\times \bbR)$ to $\dot{T}^*\bbR^n$ (resp. $\dot{T}^*(S^{n-1}\times \bbR)$). By the functorial properties of singular support,
$$\dot{SS}(\Phi(\cF)) = p_2(\dot{SS}(K)\cap p_1^{-1}(\dot{SS}(\cF))) = \chi(\dot{SS}(\cF)).$$
Here $K$ is the convolution kernel of the Radon transform, and the second equality is due to the fact that $K$ is a quantized contact transform inducing $\chi$. If $\dot{SS}(\cF)\subset \Lambda$, then $\dot{SS}(\Phi(\cF)) = \chi(\dot{SS}(\cF)) \subset \chi(\Lambda) =  \Lambda'$.

Recall $\Psi$ is the right adjoint functor of $\Phi$. For the essential subjectivity, suppose $\cG \in D^b(S^{n-1}\times \bbR)$ with $\dot{SS}(\cG)\subset \Lambda'$, then $\cF = \Psi(\cG)$ satisfies $\dot{SS}(\cF)\subset \Lambda$ and by adjunction $\Phi(\cF)\xrightarrow{\sim}\cG$. 
\end{proof}

The right hand side of the equivalene can be refined under suitable conditions. We state the refinement in a more general setting:

\begin{lem}\label{opendoesnotmatter}
Let $X$ be a manifold, and $\Lambda\subset T^{*,+} (X\times \bbR)$ a closed conic Lagrangian. Assume $\Lambda/\bbR_+$, the quotient by the dilation, is compact. There is an equivalence of categories,
$$D^b_\Lambda(X\times \bbR,\dot{T}^*(X\times \bbR)) \xrightarrow{\sim} D^b_{\Lambda}(X\times \bbR, T^{*,+}(X\times\bbR)).$$
\end{lem}

\begin{proof}\footnote{We thank St\'ephane Guillermou for correcting the proof.}
The inclusion $T^{*,+}(X\times \bbR)\subset \dot{T}^*(X\times \bbR)$ induces a quotient functor
\begin{equation}\label{quotientvsprojector}
D^b_\Lambda(X\times \bbR,\dot{T}^*(X\times \bbR)) \xrightarrow{} D^b_{\Lambda}(X\times \bbR, T^{*,+}(X\times\bbR)).
\end{equation}

Two sheaves $\cF_1, \cF_2$ are isomorphic in $D^b_{\Lambda}(X\times \bbR, T^{*,+}(X\times\bbR))$ if there is a roof
\begin{equation}\label{roof}
\begin{split}
\begin{tikzpicture}  
  \node (A){$\cF$};
  \node (C)[below of=A,node distance=1cm] {};
  \node (D)[left of=C,node distance=1.5cm] {$\cF_1$};
  \node (F)[right of=C,node distance=1.5cm]{$\cF_2$};
  \draw[->] (A) to node [swap] {$f_1$} (D);
  \draw[->] (A) to node []{$f_2$}(F);
\end{tikzpicture}
\end{split}
\end{equation}
such that (1) $\dot{SS}(\cF) \subset \Lambda\cup (\dot{T}^*(X\times \bbR)\setminus T^{*,+}(X\times\bbR))$, (2) each of $f_1,f_2$ is an isomorphism in $D^b(X\times \bbR, T^{*,+}(X\times\bbR))$. 

We first modify $\cF$ by Tamarkin's projector functor $\cP$ \cite[Proposition 2.1]{Ta}. Recall the construction.  Let $p_{12},p_3$ be the projections from $X\times \bbR\times \bbR$ to the first two coordinates and the third coordinate, and let $s: X\times \bbR\times \bbR \mapsto X\times \bbR$ by $s(x, t_1,t_2) = (x,t_1+t_2)$. Let $k_{[0,\infty)}\in D^b(\bbR)$ be the constant sheaf supported on $[0,\infty)$, then the projector is defined by $$\cP(\cF) = Rs_!(p_{12}^{-1}\cF\otimes p_3^{-1}k_{[0,\infty)}).$$
The projector has the following properties: (1) there is a canonical morphism $\cP(\cF)\rightarrow \cF$, which is an isomorphism in $D^b(X\times \bbR, T^{*,+}(X\times \bbR))$, (2) $\dot{SS}(\cP(\cF)) \cap T^{*,-}(X\times\bbR) = \emptyset$.

The second property of the projector shows $\dot{SS}(\cP(\cF))$ is excluded from $T^{*,-}(X\times\bbR)$, and we need to further argue that it does not intersect $\dot{T}^*X\times 0_{\bbR}= \dot{T}^*X \times\bbR$. This is achieve by the following microlocal cut-off lemma. Let $\pi: T^*(X\times \bbR)\rightarrow T^*X$ be the projection. Because $\Lambda/\bbR_+$ is compact, so is $\pi(\Lambda)$ in $T^*X$. It satisfying the hypothesis of \cite[Lemma 3.7]{Ta}, the lemma yields the singular support of the projected sheaf satisfies
$$SS(\cP(\cF))\cap {T}^{*,\leq 0}(X\times \bbR) \subset 0_{X\times \bbR},$$
where ${T}^{*,\leq 0}(X\times \bbR) = T^*(X\times\bbR)\setminus T^{*,+}(X\times \bbR)$. Hence $\dot{SS}(\cP(\cF)) \subset \Lambda$.

Suppose $f: \cF'\rightarrow \cF''$ is a morphism between two sheaves micro-supported along $\Lambda$, the triangle inequality (\ref{triangleinequality}) implies the cone of the morphism is also micro-supported along $\Lambda$:
$$\dot{SS}(Cone(f)) \subset \dot{SS}(\cF') \cup \dot{SS}(\cF'') \subset \Lambda.$$
Therefore $f$ is an isomorphism in $D^b(X\times \bbR,\dot{T}^*(X\times \bbR))$ (meaning $\dot{SS}(Cone(f))\cap \dot{T}^*(X\times \bbR) =\emptyset$) if and only if it is an isomorphism in $D^b(X\times \bbR,T^{*,+}(X\times\bbR))$ (meaning $\dot{SS}(Cone(f))\cap T^{*,+}(X\times \bbR) =\emptyset$).

Apply the projector to the roof (\ref{roof}). By the previous argument, each of $\cP(f_1),\cP(f_2)$ is an isomorphism in $D^b(X\times \bbR,\dot{T}^*(X\times \bbR))$ if and only if it is an isomorphism in $D^b(X\times \bbR,T^{*,+}(X\times\bbR))$. Therefore the quotient functor (\ref{quotientvsprojector}) is an equivalence. We complete the proof.  
\end{proof}

A sheaf $\cF\in D^b(X)$ is a \textit{local system} (or locally constant sheaves) if for any point on $X$ there is an open neighborhood $U$ such that the restriction $\cF|_U$ is a constant sheaf. Let $Loc(X)\subset D^b(X)$ be the subcategory of local systems on $X$. In this context, a local system is more than just a $\pi_1(X)$ representation. See the following example.

\begin{eg}\label{Hopffibration}
Let $\pi: S^3\rightarrow S^2$ be the Hopf fibration, which is a non-trivial $S^1$ fiber bundle over $S^2$ where any two fibers are linked. 

Consider $\cF = R\pi_*k_{S^3}\in D^b(S^2)$. By the functorial property of singular supports (\ref{SSfunctorial1}), $SS(\cF)\subset S^2$ and $\cF$ is a local system in $Loc(S^2)$.

At each point $p \in S^2$, the stalk $\cF_p$ is isomorphic to $\oplus H^i(\pi^{-1}(p), k) \cong k \oplus k[-1]$. However, $\cF$ and $k_{S^2}\oplus k_{S^2}[-1]$ are not isomorphic. In fact, $\cF$ can be constructed from a non-trivial extension class
$$0\rightarrow k_{S^2}\rightarrow \cF \rightarrow k_{S^2}[-1]\rightarrow 0$$
in $Ext^1(k_{S^2}[-1],k_{S^2})$.
\end{eg}

A sheaf $\cF \in D^b_{dg}(X)$ is a local system if and only if $SS(\cF)\subset 0_X$. The localization with respect to $\dot{T}^*X$ is by definition the quotient by sheaves microsupported in the $0_X$, namely the quotient of local systems:
\begin{equation}\label{quolocalsys}
D^b(X, \dot{T}^*X) \cong D^b(X)/Loc(X).
\end{equation}

\begin{lem}\label{Locretracts}
Let $X$ be smooth manifold. Then $Loc(X\times \bbR^n) \cong Loc(X)$.
\end{lem}
\begin{proof}
Let $\pi: X\times \bbR^n\rightarrow X$ be the projection. The functoriality of singular support indicates $(\pi^{-1}, R\pi_*)$ is a pair of adjoint functors between the categories of local systems. The adjunction yields a natural morphism $\pi^{-1}R\pi_*\cF\rightarrow \cF$ for $\cF\in D^b(X\times\bbR)$. Suppose $X$ is a point, then a locally constant sheaf on $\bbR^n$ is isomorphic to a constant sheaf. Therefore the natural morphism is an isomorphism, and $Loc(\bbR^n)\cong Loc(pt)$. In the general case, we take the restriction of the morphism at an each point $x\in X$, then it reduces to the case of $X$ being a point and hence is an isomorphism.
\end{proof}

\begin{thm}\label{Equivuptoloc}
Let $\Lambda\subset \dot{T}^*\bbR^n$ be a closed conic Lagrangian such that the quotient $\Lambda/\bbR_+$ is compact. Let $\chi : \dot{T}^*\bbR^n\rightarrow T^{*,+}(S^{n-1}\times \bbR)$ (see (\ref{Chiformula})) be the contact transform induced by the Radon transform, and let $\Lambda' = \chi(\Lambda)$.  There is an equivalence of categories
$$D^b_{\Lambda}(\bbR^n,\dot{T}^*\bbR^n)\cong D^b_{\Lambda'}(S^n\times \bbR,\dot{T}^*(S^n\times \bbR)).$$
It further implies
$$D^b_{\Lambda}(\bbR^n)/ Loc(pt) \cong D^b_{\Lambda'}(S^{n-1}\times \bbR)/ Loc(S^{n-1}).$$
\end{thm}

\begin{proof}
The first assertion follows the equivalences in Proposition \ref{RadonSh} and in Lemma \ref{opendoesnotmatter}.

Then we apply equation (\ref{quolocalsys}), the equivalence can be rewritten as
$$D^b_{\Lambda}(\bbR^n)/Loc(\bbR^3) \cong D^b_{\Lambda'}(S^n\times \bbR,\dot{T}^*(S^n\times \bbR))/ Loc(S^2\times \bbR).$$

Further by Lemma \ref{Locretracts} we conclude the second assertion.
\end{proof}
\begin{rmk}
Objects in $Loc(pt)$ are differential graded chain complexes of vector spaces. Moreover, each chain complex is quasi-isomorphic to its homology (with zero differentials).

Objects in $Loc(S^{n-1})$ are more complicated. For example, Guillermou \cite[Section 3]{Gu2} gives a complete classification of local systems on $S^2$. A Local system on $S^2$ is the direct sum of iterated cones of the constant sheaf on $S^2$ with shifted degrees. The first iteration gives the Hopf local system $\cF$ in Example \ref{Hopffibration}, and the next iteration is a cone between $\cF$ and a constant sheaf. The existence of the nontrivial local system $\cF$ comes from the fact that the second singular cohomology group of the sphere is non zero
$$Ext^1(k_{S^2}[-1],k_{S^2}) = Ext^2(k_{S^2},k_{S^2}) = H^2(S^2,k) = k.$$
\end{rmk}

\subsection{Legendrian knot conormal}\label{I3}

The cosphere bundle of a smooth manifold inherits a canonical contact structure. Legendrians in this contact space can be assigned with sheaf categories as invariants. One obtains a knot invariant from its Legendrian conormal bundle under this formalism. Topologists study the knot conormal in two different ambient spaces. The associated sheaf categories are related by the Radon transform.

The cotangent bundle $T^*X$ with the canonical one form is an exact symplectic manifold. Equip $X$ with any Riemannian metric, the canonical one-form restricts to a contact form on the hypersurface of covectors of any radius. We denote this cosphere bundle by $T^\infty X$. Each point in $T^\infty X$ corresponds to a unique $\bbR_+$-orbit in $\dot{T}^*X$ which is independent from the choice of the radius or the metric. Suppose $\Lambda \in T^\infty X$ is a Legendrian, associate to it a category
$$D^b_\Lambda(X):= D^b_{\bbR_+\Lambda\cup 0_X}(X),$$
here $\bbR_+\Lambda\subset \dot{T}^*X$ is the cone over $\Lambda$. A contact Hamiltonian isotopy of $T^\infty X$ lifts to a homogeneous Hamiltonian isotopy of $\dot{T}^*X$, and each Legendrian is coned to a conic Lagrangian (for this reason we use the same letter $\Lambda$ for both the Legendrian and the conic Lagrangian). Following the GKS theorem, the $D^b_\Lambda(X)$ is an Legendrian isotopy invariant. 

This framework allows us to define knot invariants using sheaves. Suppose $K\subset \bbR^3$ is a knot with possibly more than one component. Define the \textit{Legendrian knot conormal} $\Lambda_K$ to be the intersection of the knot conormal and the cosphere bundle:
$$\Lambda_K := T^*_K\bbR^3\cap T^\infty \bbR^3.$$
An isotopy of $K$ induces a Legendrian isotopy of $\Lambda_K$, and hence a Legendrian isotopy invariant of $\Lambda_K$ is in return a knot invariant. The category of sheaves micro-supported along $\Lambda_K$ is hence a knot invariant.

It is customary to consider the Legendrian knot conormal sitting inside a one-jet space for the purpose to study its Chekanov-Eliashberg dga (differential graded algebra) \cite{EENS1}. A one jet space $J^1Y \cong T^*Y\times \bbR$ over a smooth manifold $Y$ has a canonical structure $\alpha = dz-\beta$, where $z$ is coordinate on $\bbR$ and $\beta$ is the canonical one-form on $T^*Y$. Fixing the standard flat metric $\la -,-\ra$ on $\bbR^3$, and taking $T^\infty\bbR^3$ to be the cosphere bundle of unit length covectors, we have a strict contactomorphism (diffeomorphism preserving the contact form):
\begin{align*}
\varphi: T^\infty \bbR^3 &\rightarrow J^1S^2, \\
	(q,p)			&\mapsto (p, q-p\la q,p\ra, \la q, p\ra).
\end{align*}
This is the map used in \cite{EENS1}. We remark the map generalizes to other dimensions.

In general a one-jet bundle embeds into the cosphere bundle as an open submanifold:
\begin{align*}
\iota:\quad J^1Y &\rightarrow T^{\infty,+}(Y\times \bbR), \\
	(q,p,z) &\mapsto (q,z,-p,1).
\end{align*}
Here $ T^{\infty,+}(X\times \bbR)\subset T^{\infty}(X\times \bbR)$ consists of covectors whose last coordinate is positive. This notion is well defined for any choice of $T^\infty(X\times \bbR)$. Composing two maps, we get
\begin{equation}\label{consphereembedding}
\iota\circ \varphi: T^\infty \bbR^3\xrightarrow{\sim} T^{\infty,+}(S^2\times \bbR)\subset T^{\infty}(S^2\times \bbR).
\end{equation}

The ambient contact manifold where the Legendrian knot conormal lives admits two different base manifolds. Each base manifold gives rise to a sheaf category as a Legendrian invariant. We will show that these two categories are related by the Radon transform.

We first verify the geometric setup of \cite{EENS1} is the same with that of the Radon transform.

\begin{lem}\label{knotconormalunderRadon}
Let $\iota\circ \varphi: T^\infty \bbR^3\rightarrow{} T^{\infty,+}(S^2\times \bbR)$ be as defined in (\ref{consphereembedding}) and let $\chi : \dot{T}^*\bbR^3\rightarrow T^{*,+}(S^{2}\times \bbR)$  (see (\ref{Chiformula})) be the contact transform induced by the Radon transform. The two maps are compatible:
$$\chi|_{T^\infty \bbR^3} = \iota\circ \varphi.$$
\end{lem}
\begin{proof}
The restriction of $\chi$ to $T^\infty\bbR^3$ is to take the covectors of unit length. It is straight forward to verify, that
$$\chi|_{T^\infty\bbR^3}(q,p) =  (p, \la q, p\ra, -q + p\la q,p\ra, 1).$$
and that
$$\iota\circ \varphi (q,p) = \iota (p, q-p\la q,p\ra, \la q, p\ra) = (p, \la q, p\ra, -q + p\la q,p\ra, 1).$$
\end{proof}

Next we apply the theorem from the last section for the following result.
\begin{prop}\label{RTapplytoknotconormal}
Let $K\subset \bbR^3$ be a knot or a link. Let $\Lambda_K\subset T^\infty \bbR^3$ be the Legendrian knot conormal and let $\Lambda_K' = \iota\circ \varphi (\Lambda)\subset T^{\infty}(S^2\times \bbR)$ (the map is defined in (\ref{consphereembedding})). There is an equivalence of categories
$$D^b_{\Lambda_K}(\bbR^3)/ Loc(pt) \cong D^b_{\Lambda_K'}(S^2\times \bbR)/ Loc(S^2).$$
\end{prop}
\begin{proof}
The Legendrian knot conormal is topologically a torus and hence compact. In other words, the associated conic Lagrangian is horizontally compact. The equivalence follows from Theorem \ref{Equivuptoloc}, the second statement.
\end{proof}

\subsection{Simple sheaves}\label{I4}
Simple sheaves are interesting in many contexts, e.g. \cite{Gu} used simple sheaves to trivialize the Kashiwara-Schapira stack associated to a closed conic Lagrangian, and \cite{NRSSZ} proved that augmentations and simple sheaves are categorically equivalent for Legendrian knots. A natural question to ask is whether the simpleness is preserved under the Radon transform. The answer is affirmative, (in fact true for any quantized contact transform). We provide a quick proof.

\begin{prop}\label{simpleRTknotinv}
Let $\Lambda\subset \dot{T}^*\bbR^n$ be a closed conic Lagrangian. Let $\chi : \dot{T}^*\bbR^n\rightarrow T^{*,+}(S^{n-1}\times \bbR)$  (see (\ref{Chiformula})) be the contact transform induced by the Radon transform, and let $\Lambda' = \chi(\Lambda)$. 

If $\cF$ is simple along $\Lambda$, then $\Phi(\cF)$ is simple along $\Lambda'$.
\end{prop}
\begin{proof}
By definition, $\cF$ being simple along $\Lambda$ means $\mu hom(\cF,\cF)|_\Lambda = k_\Lambda$. For any $p\in \Lambda$, we have
$$\mu hom(\Phi(\cF),\Phi(\cF))_{\chi(p)} = (\chi_*\mu hom(\cF,\cF))_{\chi(p)} = \mu hom(\cF,\cF)_p = k.$$
The first equality following from Theorem \ref{ContacttransformKS} part (2). Therefore the natural morphism $k_{\Lambda'}\rightarrow \mu hom(\Phi(\cF),\Phi(\cF))|_{\Lambda'}$ is an isomorphism, and $\Phi(\cF)$ is simple along $\Lambda'$.
\end{proof}

An augmentation $\epsilon$ of a C-E dga $\cA$ (over a coefficient ring $R$) is a dga morphism $\epsilon: \cA\rightarrow k$ to the trivial dga. We consider the conormal bundle of a knot (not a link) and its C-E dga with coefficient ring $R = \bbZ[H_1(\Lambda_K)]$. Many augmentations can be understood from some representations of the knot group in $\bbR^3$ \cite{Cor13b, Ng5}. It is further generalized to a bijective correspondence between augmentations and simple sheaves on $\bbR^3$ \cite{Ga2}. The proposition implies the correspondence can be transferred to simples sheaves on $S^2\times \bbR$, where the base manifold is necessary to adopt the framework of the Nadler-Zaslow correspondence.

\section{Transforms}\label{Sec:Transforms}
Our first motivation was the observation of the similarity between the Radon transform and the projective duality. Let $\bbP^n$ be the real projective $n$-space and let $\check \bbP^n$ be its dual space. Let $[x]$ (resp. $[y]$) be the homogeneous coordinates on $\bbP^n$ (resp. $\check \bbP^n$). Let
$$P: = \{x \cdot y = 0\}$$
be a hypersurface in $\bbP^n\times\check \bbP^n$. Then $k_P\in D^b(\bbP^n\times \check \bbP^n)$ is a quantized contact transform from $(\bbP^n, \dot{T}^*\bbP^n)$ to $(\check{\bbP}^n, \dot{T}^*\check \bbP^n)$, hence induces an equivalence 
$$\Phi_P: D^b(\bbP^n, \dot{T}^*\bbP^n)\xrightarrow{\sim}D^b(\check \bbP^n, \dot{T}^*\check\bbP^n).$$
This transform is called the \textit{projective duality}, (with a self contained proof is in Proposition \ref{Projective}). We will start unwrapping this transform and eventually get to Theorem \ref{MainThm3}.

\subsection{Spherical}
To study the projective duality, we first lift the projective space to the unit sphere by the covering map $S^n\rightarrow \bbP^n$. The cotangent bundle $T^*\bbP^n$ is a $\bbZ_2$ quotient of $T^*S^n$, where the equivalence relation is $(q,p)\sim (-q,-p)$. It is similar for the dual projective space.

Let $({x}, \xi)$ (resp. $({y}, \eta)$) be coordinates on $T^*S^n$ (resp. $T^*\check{S}^n$). Regard $S^n, \check S^n$ as unit spheres in $\bbR^{n+1}$. Then
$$T^*(S^n\times \check S^n) =  \{(x,y,\xi, \eta)\,|\, \|x\| = \|y\| =1,\, \xi \cdot x =0,\, \eta\cdot y =0\textrm{ for } x,y,\xi, \eta \in \bbR^{n+1}\}.$$
Let $\hat{P} = \{ x \cdot  y =0\}\subset S^n\times \check S^n$ be the lift of $S$. Set $F(x,y) = x\cdot y$, $G_1 = \|x\|^2-1$ and $G_2 = \|y\|^2-1$. Then $T^*_{\hat S}(S^n\times \check S^n)$ is spanned by $dF$ projected to $ (dG_1)^\perp \cap (dG_2)^\perp$ along $\hat{P}$. One computes $dF =ydx + xdy$, $dG_1 = 2xdx$ and $dG_2 = 2y dy$. The projection is
$$dF - \frac{\la dF, dG_1\ra dG_1}{\|dG_1\|^2} - \frac{\la dF, dG_2\ra dG_2}{\|dG_2\|^2} = ydx +xdy.$$
Then
$$L : =\dot{SS}(k_{\hat P}) = \{ ((x,y), \ell(y,x)) \in T^*(S^n\times \check S^n) \,|\, x, y\in \bbR^{n+1}, \ell \in \bbR^*\}.$$
Let $p_1:T^*(S^n\times \check S^n)\rightarrow :T^*S^n$, $p_2:T^*(S^n\times \check S^n)\rightarrow :T^*\check S^n$ be the projections.
\begin{align*}
p_1^a|_L: L &\rightarrow \dot T^*S^n,& p_2|_L: L &\rightarrow \dot T^*S^n, \\
(x,y,\ell y, \ell x)&\mapsto  (x, -\ell y), & (x,y,\ell y, \ell x)&\mapsto  (y, \ell x).
\end{align*}
Set $(x, -\ell y) = (x,\xi)$, $(y, \ell x) = (y,\eta)$, then solve everything in terms of $x$ and $\xi$. There are two cases:
\begin{itemize}
\item If $\ell > 0$, then 
$$\ell = \|\xi\|,\, y = -\frac{\xi}{\|\xi\|},\, \eta = \frac{x}{\|\xi\|}.$$
\item If $\ell <0$, then 
$$\ell = -\|\xi\|,\, y = \frac{\xi}{\|\xi\|},\, \eta = -\frac{x}{\|\xi\|}.$$
\end{itemize}
Each case induces a contact transform from $\dot{T}^*S^n$ to $\dot{T}^*\check S^n$
\begin{align}
\begin{split}
\chi_+: (x,\xi) &\mapsto \left(-\frac{\xi}{\|\xi\|}, \frac{x}{\|\xi\|}\right),\\ \label{chi+}
\chi_-: (x,\xi) &\mapsto \left(\frac{\xi}{\|\xi\|}, -\frac{x}{\|\xi\|}\right).
\end{split}
\end{align}
Each contact transform corresponds to one of the two connected components of $L$. Each component is the singular support with zero section removed of a sheaf on $D^b(S^n\times \check S^n)$. Define
$$K_+ := k_{\{\la x,y \ra \geq 0\}}, \qquad K_- :=k_{\{\la x, y \ra\leq 0\}}.$$
We call the convolution by $K_+$ (resp. $K_-$) the positive (resp. negative) spherical duality.

\begin{prop}
The spherical duality is an equivalence of categories,
$$\Phi_\pm: D^b(S^n, \dot{T}^*S^n)\xrightarrow{\sim}D^b(\check S^n, \dot{T}^*\check S^n).$$
\end{prop}
\begin{proof}
Computed in (\ref{chi+}), $\dot{SS}(K_+)$ induces the contact transform $\chi_+$. It is similar to the proof of Theorem \ref{RTn} to check $K_+$ is a quantized contact transform and therefore $\Phi_+$ is an equivalence. Similar arguments hold for $\Phi_-$.
\end{proof}

Each spherical duality can be considered as a lift of the projective duality, and the two lifts are related. The contact transforms $\chi_+$ and $\chi_-$ are compatible with the $\bbZ_2$ action on both $T^*S^n$ and $T^*\check S^n$ -- let $\sigma$ be the non trivial element in $\bbZ_2$, the following diagrams commute.

\begin{center}
\begin{tikzpicture}
  \node (A){$\dot{T}^*S^n$};
  \node (B)[right of=A, node distance=3cm]{$\dot{T}^*S^n$};
  \node (Z)[below of=A, node distance=2cm]{};
  \node (C)[right of=Z, node distance = 1.5cm]{$\dot{T}^*\check S^n$};
  
  \node (D)[right of=B, node distance=4cm]{$\dot{T}^*S^n$};
  \node (E)[right of=C, node distance=4cm]{$\dot{T}^*\check S^n$};
  \node (F)[right of=E, node distance=3cm]{$\dot{T}^*\check S^n$};
  \draw[->] (A) to node [] {$\sigma$} (B);
  \draw[->] (A) to node [swap]{$\chi_+$}(C);
  \draw[->] (B) to node [] {$\chi_-$} (C);
  \draw[->] (D) to node [swap] {$\chi_+$} (E);
  \draw[->] (D) to node []{$\chi_-$}(F);
  \draw[->] (E) to node []{$\sigma$}(F);
\end{tikzpicture}
   \captionof{figure}{}
   \label{chiwelldefined}

\end{center}

\begin{prop}\label{Projective}
The projective duality is an equivalence of categories.
\end{prop}
\begin{proof}
Recall $P: = \{x \cdot y = 0\}\subset \bbP^n\times \check \bbP^n$. Let $\chi$ be the symplectomorphism induced by $\dot{T}^*\bbP^n \xleftarrow{p_1^a} \dot{SS}(k_P) \xrightarrow{p_2} \dot{T}^* \check \bbP^n$. The commutative diagrams in Figure \ref{chiwelldefined} imply that $\chi$ is well defined. Then the assertion follows from Theorem \ref{ContacttransformKS}.
\end{proof}
 
The rest of the subsection is to relate the spherical and projective dualities by polarization. 
\begin{defn}\label{polardef}
Let $\tau$ be the antipodal map on the sphere, we define the polarization functor to be
$$pl: D^b(S^n)\rightarrow D^b(S^n), \qquad \cF \mapsto \cF\oplus \tau^{-1}\cF.$$
Morphisms are induced by compositions with identity and $\tau$. We define the same for $\check S^n$.
\end{defn}
\begin{lem}\label{SphereProjI}
Let $\Phi_{\hat{P}}$ be the convolution by $k_{\hat{P}}\in D^b(S^n\times \check S^n)$ and $\Phi_\pm$ the spherical duality. 
For any $\cF \in D^b_{}(S^n)$, there is an isomorphism in $D^b(\check S^n, \dot{T}^*\check S^n)$,
$$pl\circ \Phi_\bullet (\cF) \cong \Phi_{\hat{P}}(\cF), \quad \text{for}\;\; \bullet = \pm.$$
\end{lem}
\begin{proof}
First, for any $\cF \in D^b(S^n)$, there is an isomorphism $\tau^{-1}\Phi_+(\cF)\cong \Phi_-(\cF)$ because
$$\tau^{-1}R\pi_{2!}(K_+\otimes \pi_1^{-1}\cF) = R\pi_{2!}(id_{S^n}\times \tau_{\check S^n})^{-1}(K_+ \otimes \pi_1^{-1}\cF) = R\pi_{2!}(K_- \otimes \pi_1^{-1}\cF).$$
Since $\tau^2 = id_{\check S^n}$, there is also $\tau^{-1}\Phi_-(\cF)\cong \Phi_+(\cF)$.

Recall $\hat{P} = \{ x \cdot y =0\}\subset S^n\times \check S^n$. There is a short exact sequence of sheaves on $S^n\times \check S^n$:
$$0\rightarrow k_{S^n\times \check S^n} \rightarrow K_+ \oplus K_- \rightarrow k_{\hat{P}}\rightarrow 0.$$
Because $k_{S^n\times \check S^n}$ is the constant sheaf on the product space, convolution by this kernel ends up with a locally constant sheaf. Hence the canonical morphism
$$\Phi_+ (\cF) \oplus \Phi_-(\cF) \rightarrow \Phi_{\hat{P}}(\cF)$$
is an isomorphism in $D^b(\check S^n, \dot{T}^*\check S^n)$.

Finally we have $pl\circ \Phi_+ (\cF) = \Phi_+(\cF) \oplus \tau^{-1}\Phi_+(\cF) = \Phi_+ (\cF) \oplus \Phi_-(\cF) \xrightarrow{\sim} \Phi_{\hat{P}}(\cF)$ in $D^b(\check S^n, \dot{T}^*\check S^n)$. Similar arguments hold for $\Phi_-$.
\end{proof}

\begin{lem}\label{LiftTransform}
Suppose $f_1: X'\rightarrow X$ and $f_2: Y'\rightarrow Y$ are smooth maps of manifolds. Let $\pi_i$ be projections on $X\times Y$ and $\hat{\pi}_i$ be projections on $X'\times Y'$. Suppose $K\in D^b(X\times Y)$ and let $K' = (f_1\times f_2)^{-1}K$. Then
$$\Phi_{K'}\circ f_1^{-1}(\cF) =  f_2^{-1} R{\pi}_{2!} R(f_1\times id)_{!}(f_1\times id)^{-1}(K\otimes \pi_{1}^{-1}\cF).$$
\end{lem}
\begin{proof}
Consider the following diagram.
\begin{center}
\begin{tikzpicture}
  \node (A){$X'$};
  \node (B)[right of=A, node distance=2.7cm]{$X'\times Y'$};
  \node (C)[right of=B, node distance=3.5cm]{$X'\times Y'$};
  \node (D)[right of=C, node distance=2.7cm]{$Y'$};
  \node (E)[below of=A, node distance=2cm]{$X$};
  \node (F)[below of=B, node distance=2cm]{$X\times Y$};
  \node (G)[below of=C, node distance=2cm]{$X'\times Y$};
  \node (H)[below of=D, node distance=2cm]{$Y$};
  \draw[<-] (A) to node [] {$\hat{\pi}_1$} (B);
  \draw[<-] (B) to node []{$=$}(C);
  \draw[->] (C) to node []{$\hat{\pi}_2$}(D);
  \draw[<-] (E) to node [swap] {$\pi_1$} (F);
  \draw[<-] (F) to node [swap]{$f_1\times id$}(G);
  \draw[->] (G) to node [swap]{$\tilde{\pi}_2$}(H);
  \draw[->] (A) to node [swap]{$f_1$}(E);
  \draw[->] (B) to node [swap]{$f_1\times f_2$}(F);
  \draw[->] (C) to node [swap]{$id\times f_2$}(G);
  \draw[->] (D) to node []{$f_2$}(H);
\end{tikzpicture}
\end{center}
Note that $\tilde{\pi}_2 = {\pi}_2 \circ (p_1\times id)$. Then we have
\begin{align*}
\Phi_{K'}\circ f_1^{-1}(\cF) &= R\hat\pi_{2!}({K'}\otimes \hat\pi_1^{-1}(p_{1}^{-1}\cF)) \\
	&= R\hat\pi_{2!}((f_1\times f_2)^{-1}K\otimes \hat\pi_1^{-1}(p_{1}^{-1}\cF))	 &&[K'=(f_1\times f_2)^{-1}K ]\\
	&= R\hat\pi_{2!}((f_1\times f_2)^{-1}K\otimes (f_1\times f_2)^{-1}(\pi_{1}^{-1}\cF))  &&[\text{Left square}] \\
	&= R\hat\pi_{2!}(f_1\times f_2)^{-1}(K\otimes \pi_{1}^{-1}\cF)  &&\text{\cite[Prop 2.3.5]{KS}}\\
	&= R\hat\pi_{2!}(id\times f_2)^{-1}(f_1\times id)^{-1}(K\otimes \pi_{1}^{-1}\cF) &&[\text{Middle square}]\\
	&= f_2^{-1} R\tilde{\pi}_{2!} (f_1\times id)^{-1}(K\otimes \pi_{1}^{-1}\cF)    &&[\text{Right square}, \text{\cite[Prop 2.6.7]{KS}}] \\
	&= f_2^{-1} R{\pi}_{2!} R(f_1\times id)_{!}(f_1\times id)^{-1}(K\otimes \pi_{1}^{-1}\cF) &&[\tilde{\pi}_2 = {\pi}_2 \circ (p_1\times id)]
\end{align*}

\end{proof}

Let $D^b_{\bbZ_2}(X)\subset D^b(X)$ be the full subcategory of equivariant sheaves with respect to the $\bbZ_2$ action. This category is canonically identified with $D^b(\bbP^n)$. Suppose $p: S^n\rightarrow \bbP^n$ is the covering map, then $p^{-1}\cF$ is a $\bbZ_2$-equivariant sheaf on $S^n$ for any $\cF\in D^b(\bbP^n)$. Moreover, $Rp_!p^{-1}(\cF) = \cF^{\oplus 2}$.

\begin{lem}\label{SphereProjII}
Let $p_1: S^n\rightarrow \bbP^n$, $p_2: \check S^n\rightarrow \check\bbP^n$ be the covering maps. For $\cF\in D^b(\bbP^n)$, there is an isomorphism in $D^b(\check S^n)$:
$$\Phi_{\hat{P}}\circ p_1^{-1}(\cF) \cong p_2^{-1}\circ \Phi_{P}(\cF)^{\oplus 2}.$$
\end{lem}
\begin{proof}
Apply Lemma \ref{LiftTransform}, we have
$$\Phi_{\hat{P}}\circ p_1^{-1}(\cF) = p_2^{-1} R{\pi}_{2!} R(p_1\times id)_{!}(p_1\times id)^{-1}(k_P\otimes \pi_{1}^{-1}\cF).$$
Because $p_1\times id$ is a two-to-one covering map, we have $R(p_1\times id)_!(p_1\times id)^{-1}(\cG) = \cG^{\oplus 2}$ for $\cG \in D^b(X\times Y)$. Hence 
$$\Phi_{\hat{P}}\circ p_1^{-1}(\cF) = p_2^{-1}R\pi_{2!}(k_P\otimes \pi_{1}^{-1}\cF)^{\oplus 2} = p_2^{-1}\circ \Phi_{P}(\cF)^{\oplus 2}.$$
\end{proof}

\begin{prop}\label{polarizationresult}
Let $\Phi_P$ be the projective duality for $\bbZ_2$-equivariant sheaves, $\Phi_\pm$ the spherical dualities. For $\cF \in D^b_{\bbZ_2}(S^n)$, there is an isomorphism in $D^b_{\bbZ_2}(\check S^n, \dot{T}^*\check S^n)$:
$$pl\circ \Phi_\bullet (\cF) \cong \Phi_P(\cF)^{\oplus 2}, \quad \text{for}\;\; \bullet = \pm.$$
\end{prop}
\begin{proof}
The assertion follows from Lemma \ref{SphereProjI} and Lemma \ref{SphereProjII}, and that every $\bbZ_2$-equivariant sheaf can be expressed as $p_i^{-1}\cF, i =1,2$.
\end{proof}

\subsection{Fourier-Sato}
In this subsection, we show that the spherical duality can be obtained by restricting the Fourier-Sato transform to open subsets.
Recall the Fourier-Sato transform from \cite{KS}. Let $(x,y)$ be coordinates of $\bbR^{n}\times \check \bbR^{n}$. Define subsets
\begin{align*}
A = \{(x,y)\in \bbR^{n}\times \check \bbR^{n} | \la x,y\ra \leq 0\};\\
B = \{(x,y)\in \bbR^{n}\times \check \bbR^{n} | \la x,y\ra \geq 0\}.
\end{align*}
The Fourier-Sato transform is the convolution by $k_A$, and the inverse Fourier-Sato transform is the convolution by $k_B[n]$. 

Let $D^b_{conic}(X)\subset D^b(X)$ be the full subcategory of conic sheaves. The Fourier-Sato transform is a equivalence between $D^b_{conic}(\bbR^n)$ and $D^b_{conic}(\check\bbR^n)$.

Euclidean space $\bbR^{n}$ removing $\{0\}$ is diffeomorphic to $S^{n-1}\times \bbR$. Because $\bbR_{> 0}$ acts freely and transitively on $S^{n-1}\times \bbR$ along the second component, there is an equivalence 
$$ D^b_{conic}(\bbR^n\setminus \{0\})\cong D^b(S^{n-1}).$$
To be more specific, suppose $p_1: S^{n-1}\times \bbR \rightarrow S^{n-1}$ (resp.  $p_1: \check S^{n-1}\times \bbR \rightarrow \check S^{n-1}$) is the projection, then every conic sheaf on $\bbR^n\setminus\{0\}$ is isomorphic to $p_1^{-1}\cF$ for some $\cF \in D^b(S^{n-1})$. The equivalence is given by the pair of adjoint functors $(p_1^{-1},Rp_{1*})$.

Consider the open embeddings $j_1: \bbR^n\setminus \{0\}\rightarrow \bbR^n$ and $j_2: \check \bbR^n\setminus \{0\}\rightarrow \check \bbR^n$. The kernel of the negative spherical duality is a restriction of the kernel of Fourier-Sato transform: $(p_1\times p_2)^{-1}K_- = (j_1\times j_2)^{-1}k_A$. Similarly the kernel of the positive spherical duality is the restriction of the kernel of the inverse Fourier-Sato transform up to a degree shift. By \cite[Example 3.3.6]{KS}, the convolution by each of $K_\pm$ is an equivalence between $D^b(S^n)$ and $D^b(\check S^n)$.

The equivalence descends to the quotient categories by local systems. It suffices to show the convolution by $K_-$ is an equivalence between $Loc(S^n)$ and $Loc(\check S^n)$ (similar to $K_+$). The singular support of a local system on $S^n$ is contained in $0_{S^n}$, and hence empty in $\dot T^*S^n$. We applying the contact transform (\ref{chi+}), the transformed sheaf also has empty singular support in $\dot T^*\check S^n$, which is therefore a local system. The equivalence follows from the fact that the inverse functor is also a convolution, given by the kernel $k_{\{y\cdot x <0\}}\in D^b(\check S^n\times S^n)$ (\cite{SKK} or \cite[Example 3.3.6]{KS}). Taking the quotient of the local systems, we recover the spherical duality:
$$D^b(S^n, \dot{T}^*S^n) \cong D^n(\check S^n, \dot{T}^*\check S^n).$$ 

We finally remark that the inverse of $\Phi_{K_-}: D^b(\check S^n)\xrightarrow{\sim} D^b( S^n)$ is the convolution by $K_-^{-1} := k_{\{\la y, x\ra <0\}}$, instead of $K_+ = k_{\{\la y, x \ra \geq 0\}}$. The images by $\Phi_{K_-^{-1}}$ and $\Phi_{K_+}$ are isomorphic up to a degree shift in $D^b(S^n, \dot{T}^*S^n)$, which may not hold in $D^b(S^n)$. Consider the short exact sequence:
$$0\rightarrow K_-^{-1} \rightarrow k_{\check S^n\times S^n}\rightarrow K_+\rightarrow 0.$$
Therefore $K_-^{-1}[1]\cong K_+ \in D^b(\check S^n\times  S^n, \dot{T}^*(\check S^n\times  S^n))$ and the assertion follows.

\subsection{Radon}
We use the spherical duality to bridge the Radon transform and the projective duality. We  introduce the restriction of a quantized contact transform, beginning with the following lemma.

\begin{lem}\label{restriction}
Suppose $K\in D^b(X\times Y)$ is a quantized contact transform from $(X,\Omega_X)$ to $(Y,\Omega_Y)$. Suppose $X'$ (resp. $Y'$) is an open subset of $X$ (resp. $Y$), and let $j_1: X'\rightarrow X$ (resp. $j_2: Y'\rightarrow Y$) be the open embedding. Then $K':= (j_1\times j_2)^{-1}K \in D^b(X'\times Y')$ is also a quantized contact transform from $(X',\Omega_{X'})$ to $(Y',\Omega_{Y'})$ for open subsets $\Omega_{X'}\subset \Omega_X \cap T^*X'$ and $\Omega_{Y'} \subset \Omega_Y\cap T^*Y'$ determined by $K'$. In particular, the convolution by $K'$ is an equivalence of categories:
$$\Phi_{K'}: D^b(X', \Omega_{X'})\xrightarrow{\sim} D^b(Y', \Omega_{Y'}).$$
\end{lem}
\begin{proof}
We first define two the open subsets $\Omega_{X'}, \Omega_{Y'}$ from $K'$, and show that the singular support of $K'$ defines a contact transform between the two. Let
\begin{align*} 
L &:= \dot{SS}(K) \subset \dot{T}^*(X\times Y), \\
L' &:= \dot{SS}(K') \subset  \dot{T}^*(X'\times Y').
\end{align*}
Apply Proposition \ref{SSfunctorial} to $j: = j_1\times j_2: X'\times Y' \rightarrow X\times Y$. Then
$$L' = \dot{SS}(K') =\dot{SS}(j^{-1}K) = {j_d}j_\pi^{-1}(\dot{SS}(K)) = {j_d}j_\pi^{-1}L.$$
Because $j$ is an open embedding, ${j_d}j_\pi^{-1}$ is just the restriction map. Hence we can identify $L'$ as a submanifold of $L$.
\begin{equation}\label{L'constraints}
L' = L \cap (\dot{T}^*X'\times \dot{T}^*Y') \subset (\dot{T}^*X'\cap \Omega_X^a)\times (\dot{T}^*Y'\cap \Omega_{Y}).
\end{equation}
Let $p_i, i=1,2$ be the projection from $\dot{T}^*X\times \dot{T}^*Y$ to the $i$-th component. Define 
$$\Omega_{X'} := (p_1)^a|_L(L'), \quad \Omega_{Y'} := p_2|_L(L')$$
It follows (\ref{L'constraints}) that $\Omega_{X'}\subset \Omega_X \cap T^*X'$ and $\Omega_{Y'} \subset \Omega_Y\cap T^*Y'$. Moreover, the diffeomorphisms $\Omega_X\cong L\cong \Omega_Y$ implies the diffeomorphsims $\Omega_{X'}\cong L'\cong \Omega_{Y'}$.

Next we apply Theorem \ref{ContacttransformKS} to prove that $\Phi_{K'}: D^b(X', \Omega_{X'})\rightarrow D^b(Y', \Omega_{Y'})$ is an equivalence of categories. First, a sheaf being cohomologically constructible is a local condition. Therefore $K$ is cohomologically constructible implies that $K'$ is cohomologically constructible. The second condition is evident by the construction of $\Omega_{X'}$ and $\Omega_{Y'}$. Finally we need to verify that there is an isomorphism:
$$k_{L'}\rightarrow \mu hom (K',K')|_{L'}.$$
Let $M = X\times Y$ and let $p_1,p_2$ be the projections from $M\times M$ to the first and second components. Similarly define $q_1,q_2$ for $M'\times M'$, where $M' = X'\times Y'$. Let $j\times j: M' \times M' \rightarrow M\times M$ be the open embedding induced by $j$. Note $j^! = j^{-1}$ because $j$ is open.
\begin{align*}
\mu hom (K',K') &= \mu_{\Delta_{M'}} R\cH om_{M'\times M'} (q_2^{-1}K',q_1^!K') \\
		&= \mu_{\Delta_{M'}} R\cH om_{M'\times M'} (q_2^{-1}j^{-1}K,q_1^!j^!K) \\
		&= \mu_{\Delta_{M'}} R\cH om_{M'\times M'} ((j\times j)^{-1}p_2^{-1}K,(j\times j)^{!}p_2^{!}K) \\
		&= \mu_{\Delta_{M'}} (j\times j)^!R\cH om_{M\times M} (p_2^{-1}K,p_2^{!}K)\\
		&= \mu_{\Delta_{M}} R\cH om_{M\times M} (p_2^{-1}K,p_2^{!}K)|_{T^*M'}.
\end{align*}
The last line follows from \cite[Proposition 4.3.5]{KS}, which states $\mu_{\Delta_{M'}} (j\times j)^!\cG = \mu_{\Delta_M}\cG|_{T^*M'}$ for $\cG \in D^b(M\times M)$ (bottom row of the diagram in the proposition, \textit{loc.cit.}). Because the canonical morphism $k_{L}\rightarrow \mu hom (K,K)|_L$ is an isomorphism, so is $k_{L'}\rightarrow \mu hom (K',K')|_{L'}$. We complete the proof.
\end{proof}

\begin{defn}\label{resdef}
Following the notation in Lemma \ref{restriction}, we say $\Phi_{K'}: D^b(X', \Omega_{X'})\xrightarrow{\sim} D^b(Y', \Omega_{Y'})$ is a \textit{restriction} of $\Phi_{K}: D^b(X, \Omega_{X})\xrightarrow{\sim} D^b(Y, \Omega_{Y}).$
\end{defn}

\begin{rmk}
The open subsets $\Omega_{X'}$ and $\Omega_{Y'}$ are determined by the pair of open embeddings $j_1$ and $j_2$. In general it is not known if $\Omega_{X'} = \Omega_X \cap \dot T^*X'$ or $\Omega_{Y'} = \Omega_Y \cap \dot T^*Y'$. 
\end{rmk}

\begin{prop}\label{CorRes}
Following the notation in Lemma \ref{restriction}, if $\Phi_{K'}$ is a restriction of $\Phi_K$, then
$j_2^{-1}\circ\Phi_K = \Phi_{K'}\circ j_1^{-1}$. In other words, the following diagram commutes.
\begin{center}
\begin{tikzpicture}
  \node (A){$D^b(X, \Omega_X)$};
  \node (B)[right of=A, node distance=4cm]{$D^b(Y, \Omega_Y)$};
  \node (D)[below of=A, node distance=2cm]{$D^b(X', \Omega_{X'})$};
  \node (E)[below of=B, node distance=2cm]{$D^b(Y', \Omega_{Y'})$};
  \draw[->] (A) to node [] {$\Phi_{K}$} (B);
  \draw[->] (A) to node [swap]{$j_1^{-1}$}(D);
  \draw[->] (B) to node []{$j_2^{-1}$}(E);
  \draw[->] (D) to node []{$\Phi_{K'}$}(E);
\end{tikzpicture}
\end{center}
\end{prop}
\begin{proof}
For any sheaf $\cF$ on $X$, we have
$$\Phi_{K'}\circ j_1^{-1} (\cF)=j_2^{-1} R{\pi}_{2!} R(j_1\times id)_{!}(j_1\times id)^{-1}(K\otimes \pi_{1}^{-1}\cF)  = j_2^{-1} R{\pi}_{2!} (K\otimes \pi_{1}^{-1}\cF) = j_2^{-1}\circ\Phi_K(\cF).$$
The first step follows from Lemma \ref{LiftTransform} and the second step is because $j_1\times id$ is an open embedding.
\end{proof}

Now we state the relation between the Radon transform and the spherical duality.
\begin{prop}\label{RaS+}
The Radon transform is a restriction of the negative spherical duality.
\end{prop}

\begin{proof}
Consider the open embeddings
\begin{align*}
j_1: \bbR^n&\rightarrow S^n,  & j _2: S^{n-1}\times \bbR &\rightarrow \check S^n,\\
x &\mapsto \frac{1}{\sqrt{x^2+1}}(x,-1), &(\hat n , r) &\mapsto \frac{1}{\sqrt{r^2+1}}(\hat n, r).
\end{align*}
Let $K_-$ be the kernel of the negative spherical duality. We compute the restriction $(j_1\times j_2)^{-1}K_-$. It is the constant sheaf supported on a subset of $\bbR^n\times (S^{n-1}\times \bbR)$:
$$\left\{\frac{1}{\sqrt{x^2+1}}(x,-1) \cdot \frac{1}{\sqrt{r^2+1}}(\hat n, r) \leq 0\right\}.$$
It simplifies to
$$\{x\cdot n \leq r\}.$$
This proves the statement of the proposition.
\end{proof}

\begin{rmk}\label{RaS-}
The Radon transform is also a restriction of the negative spherical duality, given by
$$j_1(x) = \frac{1}{\sqrt{x^2+1}}(-x,1), \quad j _2(n,r) = \frac{1}{\sqrt{r^2+1}}(\hat n, r).$$
\end{rmk}

\begin{rmk}
The restrictions defined in Proposition \ref{RaS+} and Remark \ref{RaS-} are compatible with the $\bbZ_2$ action on the sphere. Hence we can think of the Radon transform as an ``affine'' projective duality.
\end{rmk}

Here is a quick application. Let $K\subset \bbR^3\subset S^3$ be a knot or a link, and $\Lambda_K$ the associated Legendrian conormal. We have the following commutative diagram:
\begin{center}
\begin{tikzpicture}
  \node (A){$D^b_{\Lambda_K}(S^3)/Loc(S^3)$};
  \node (B)[right of=A, node distance=5cm]{$D^b_{\Lambda_{K'}}(\check S^3)/Loc(\check S^3)$};
  \node (D)[below of=A, node distance=2cm]{$D^b_{\Lambda_K}(\bbR^3)/Loc(pt)$};
  \node (E)[below of=B, node distance=2cm]{$D^b_{\Lambda_{K'}}(S^2\times \bbR)/Loc(S^2)$};
  \draw[->] (A) to node [] {$\cong$} (B);
  \draw[->] (A) to node [swap]{$j_1^{-1}$}(D);
  \draw[->] (B) to node []{$j_2^{-1}$}(E);
  \draw[->] (D) to node []{$\cong$}(E);
\end{tikzpicture}
\end{center}
The vertical arrows are given by Proposition \ref{CorRes} and Proposition \ref{RaS+}. The bottom isomorphism comes from Corollary \ref{RTapplytoknotconormal}. The top isomorphism can be similarly proven. The simpleness passes through the diagram. It is proven in \cite{Ga2} that the moduli set of simple sheaves are isomorphic for the two categories on the left column, and therefore for all four categories.

\subsection{Fourier-Tamarkin} We first recall Tamarkin's original construction in \cite{Ta}. Consider the maps:
\begin{align*}
p_{13}: \bbR^n \times \check \bbR^n \times \bbR \times \bbR &\rightarrow \bbR^n\times \bbR, &\quad(x,y,s,c)&\mapsto (x,s);\\
p_{124}: \bbR^n \times \check \bbR^n \times \bbR \times \bbR&\rightarrow \bbR^n \times \check \bbR^n \times \bbR, &\quad(x,y,s,c)&\mapsto (x,y,c);\\
\pi:  \bbR^n \times \check \bbR^n \times \bbR \times \bbR &\rightarrow \check \bbR^n \times \bbR, &\quad(x,y,s,c)&\mapsto (y,s+c).
\end{align*}
Let $({x},{y},c)$ be coordinates of $\bbR^n \times \check \bbR^n \times \bbR$, and let $\cG = k_{\{\la{x},{y}\ra +c\geq 0\}}\in D^b(\bbR^n \times \check \bbR^n \times \bbR)$. The Fourier-Tamarkin transform from $D^b(\bbR^n\times \bbR, T^{*,+}(\bbR^n\times \bbR))$ to $D^b(\check \bbR^n\times \bbR, T^{*,+}(\check \bbR^n\times \bbR))$ is given by
$$\Phi_{FT}(\cF) = R\pi_{!}(p_{13}^{-1}\cF\otimes p_{124}^{-1}\cG).$$ 
It is an equivalence of categories.

\begin{lem} Let $(x,s)$ be coordinates on $\bbR^n\times \bbR$ and $(y,t)$ be coordinates on $\check \bbR^n\times \bbR$. The Fourier-Tamarkin transform
$$\Phi_{FT}: D^b(\bbR^n\times \bbR, T^{*,+}(\bbR^n\times \bbR)) \rightarrow D^b(\check \bbR^n\times \bbR, T^{*,+}(\check \bbR^n\times \bbR)).$$
is a convolution with the kernel
$$K_{FT} = k_{\{\la x, y\ra +t-s \geq 0\}}.$$
\end{lem}
\begin{proof} We use the subscript to indicate the coordinates of different copies of $\bbR$. Consider the sheaf $\tilde{K}\in D^b(\bbR^n\times \check \bbR^n \times \bbR_s\times \bbR_t\times \bbR_c)$ defined as 
$$\tilde{K} = k_{\{\la x, y\ra + c \geq 0,\, t= s+c\}}.$$
Let $q_{\bullet}$ be the projection from $\bbR^n\times \check \bbR^n \times \bbR_s\times \bbR_t\times \bbR_c$ to the components indexed by the bullet. The Fourier Tamarkin transform is equivalent to
$$\Phi_{FT}(\cF) = Rq_{24!}\left(q_{13}^{-1}\cF\otimes \tilde{K}\right).$$

In $supp(\tilde{K})$, the value of of $c$ is uniquely determined by $s$ and $t$. Therefore $q_{1234}$ is proper on $supp(\tilde{K})$, and $Rq_{1234!}\tilde{K} = K_{FT}$.

Let $p_{\bullet}$ be the projection from $\bbR^n\times \check \bbR^n \times \bbR_s\times \bbR_t$ to the components index by the subscript. In particular there is $q_{\bullet} = p_\bullet \circ q_{1234}$. Recall a general formula $Rf_!(f^{-1}\cG \otimes \cH) = \cG \otimes Rf_! \cH$. Take $f = q_{1234}, \cG = p^{-1}_{13}\cF$, and $\cH = \tilde{K}$, we have
$$\Phi_{FT}(\cF) = Rp_{24!} (Rp_{13}^{-1}\cF \otimes Rq_{1234!}\tilde{K}) = Rp_{24!}(Rp_{13}^{-1}\cF \otimes K_{FT}).$$
We complete the proof.
\end{proof}

\begin{rmk}
The inverse Fourier-Tamarkin transform can be calculated is a similar way. It is the convolution by $$K^{-1}_{FT} = k_{\{-\la x, y\ra + s-t \geq 0\}}.$$
\end{rmk}

\begin{prop}\label{FTRa}
The Fourier-Tamarkin transform is a restriction of the Radon transform.
\end{prop}
\begin{proof}
Consider the open embeddings
\begin{align*}
j_1: \bbR^n \times\bbR &\rightarrow \bbR^{n+1}, & j_2: \check \bbR^n\times\bbR &\rightarrow S^{n}\times \bbR, \\
		(x,s) &\mapsto (-x,-s), & (y,t) &\mapsto \frac{1}{\sqrt{y^2+1}}((y,-1), t).
\end{align*}
It is straight forward to check the restriction of the Radon transform kernel is the Fourier-Tamarkin kernel.
\end{proof}

\begin{rmk}
The composition of Proposition \ref{RaS+} and Proposition \ref{FTRa} gives an explicit relation between the Fourier-Tamarkin transform and the spherical duality:
$$j_1: \bbR^n \times \bbR\rightarrow S^{n+1},\qquad (x,s)\mapsto \frac{1}{\sqrt{x^2+s^2+1}}(-x,-s,-1),$$
and 
$$j_2: \check \bbR^n \times \bbR\rightarrow \check S^{n+1},\qquad (y,t)\mapsto \frac{1}{\sqrt{y^2+1+t^2}}(y,-1,t).$$
The overwhelming minus signs are because of the different conventions in \cite{KS} and \cite{Ta}.
\end{rmk}

Finally, we put the results into the diagram in Theorem \ref{MainThm3}.

\begin{proof}[Proof of Theorem \ref{MainThm3}]
It follows from Proposition \ref{RaS+}, Proposition \ref{FTRa} and Proposition \ref{polarizationresult}.
\end{proof}

\end{document}